\documentclass[reqno,11pt]{amsart}
\usepackage[margin=1in]{geometry}

\usepackage{amsthm, amsmath, amssymb, bm}

\usepackage[dvipsnames]{xcolor}

\usepackage{tikz}

\usepackage{microtype} 

\usepackage[utf8]{inputenc}
\usepackage[T1]{fontenc}

\usepackage[textsize=small,backgroundcolor=orange!20]{todonotes}

\usepackage[hidelinks]{hyperref}
\usepackage{url}

\usepackage[noabbrev,capitalize]{cleveref}
\crefname{equation}{}{}

\usepackage{xcolor,graphics}

\newtheorem{theorem}{Theorem}[section]
\newtheorem{proposition}[theorem]{Proposition}
\newtheorem{lemma}[theorem]{Lemma}

\newtheorem{conjecture}[theorem]{Conjecture}
\newtheorem*{question*}{Question} \Crefname{question}{Question}{Questions}

\theoremstyle{definition}

\theoremstyle{remark}
\newtheorem*{remark}{Remark}

\newcommand{\mb}{\mathbb}
\newcommand{\on}{\operatorname}

\title{An improved bound on the least common multiple of polynomial sequences}

\author[Sah]{Ashwin Sah}
\address{Massachusetts Institute of Technology, Cambridge, MA 02139, USA}
\email{asah@mit.edu}

\date{\today}

\begin{document}

\begin{abstract}
Cilleruelo conjectured that if $f\in\mb Z[x]$ of degree $d\ge 2$ is irreducible over the rationals, then $\log\on{lcm}(f(1),\ldots,f(N))\sim(d-1)N\log N$ as $N\to\infty$. He proved it for the case $d = 2$. Very recently, Maynard and Rudnick proved there exists $c_d > 0$ with $\log\on{lcm}(f(1),\ldots,f(N))\gtrsim c_d N\log N$, and showed one can take $c_d = \frac{d-1}{d^2}$. We give an alternative proof of this result with the improved constant $c_d = 1$. We additionally prove the bound $\log\on{rad}\on{lcm}(f(1),\ldots,f(N))\gtrsim\frac{2}{d}N\log N$ and make the stronger conjecture that $\log\on{rad}\on{lcm}(f(1),\ldots,f(N))\sim (d-1)N\log N$ as $N\to\infty$.
\end{abstract}

\maketitle

\section{Introduction}\label{sec:introduction}
If $f\in\mb Z[x]$, let $L_f(N) = \on{lcm}\{f(n): 1\le n\le N\}$, where say we ignore values of $0$ in the LCM and set the LCM of an empty set to be $1$. It is a well-known consequence of the Prime Number Theorem that
\[\log\on{lcm}(1,\ldots,N)\sim N\]
as $N\to\infty$. Therefore, a similar linear behavior should occur if $f$ is a product of linear polynomials. See the work of Hong, Qian, and Tan \cite{HQT12} for a more precise analysis of this case. On the other hand, if $f$ is irreducible over $\mb Q$ and has degree $d\ge 2$, $\log L_f(N)$ ought to grow as $N\log N$ rather than linearly. In particular, Cilleruelo \cite{C11} conjectured the following growth rate.
\begin{conjecture}[{\cite{C11}}]\label{conj:lcm}
If $f\in\mb Z[x]$ is irreducible over $\mb Q$ and has degree $d\ge 2$, then
\[\log L_f(N)\sim (d-1)N\log N\]
as $N\to\infty$.
\end{conjecture}
He proved this for $d = 2$. As noted in \cite{MR19}, his argument demonstrates
\begin{equation}\label{eq:upper-bound}\tag{1.1}
\log L_f(N)\lesssim (d-1)N\log N.
\end{equation}
Hong, Luo, Qian, and Wang \cite{HLQW13} showed that $\log L_f(N)\gg N$, which was for some time the best known lower bound. Then, very recently, Maynard and Rudnick \cite{MR19} provided a lower bound of the correct magnitude.
\begin{theorem}[{\cite[Theorem~1.2]{MR19}}]\label{thm:maynard-rudnick}
Let $f\in\mb Z[x]$ be irreducible over $\mb Q$ with degree $d\ge 2$. Then there is $c = c_f > 0$ such that
\[\log L_f(N)\gtrsim cN\log N.\]
\end{theorem}
The proof given produces $c_f = \frac{d-1}{d^2}$, although a minor modification produces $c_d = \frac{1}{d}$. We prove the following improved bound, which in particular recovers \cref{conj:lcm} when $d = 2$. It also does not decrease with $d$, unlike the previous bound.
\begin{theorem}\label{thm:main}
Let $f\in\mb Z[x]$ be irreducible over $\mb Q$ with degree $d\ge 2$. Then
\[\log L_f(N)\gtrsim N\log N.\]
\end{theorem}
It is also interesting to consider the problem of estimating $\ell_f(N) = \on{rad}\on{lcm}(f(1),\ldots,f(n))$. (Recall that $\on{rad}(n)$ is the product of distinct primes dividing $n$.) It is easy to see that the proof of \cref{thm:maynard-rudnick} that was given in \cite{MR19} implies
\[\log\ell_f(N)\gtrsim c_dN\log N\]
for the same constant $c_d = \frac{d-1}{d^2}$ (or $c_d = \frac{1}{d}$ after slight modifications). We demonstrate an improved bound.
\begin{theorem}\label{thm:main-rad}
Let $f\in\mb Z[x]$ be irreducible over $\mb Q$ with degree $d\ge 2$. Then
\[\log\ell_f(N)\gtrsim\frac{2}{d}N\log N.\]
\end{theorem}
We conjecture that the radical of the LCM should be the same order of magnitude as the LCM.
\begin{conjecture}\label{conj:rad}
If $f\in\mb Z[x]$ is irreducible over $\mb Q$ with degree $d\ge 2$, then
\[\log\ell_f(N)\sim (d-1)N\log N\]
as $N\to\infty$.
\end{conjecture}
Finally, we note that \cref{thm:main-rad} proves \cref{conj:rad} for $d = 2$.

In a couple of different directions, Rudnick and Zehavi \cite{RZ19} have studied the growth of $L_f$ along a shifted family of polynomials $f_a(x) = f_0(x) - a$, and Cilleruelo has asked for similar bounds in cases when $f$ is not irreducible as detailed by Candela, Ru\'{e}, and Serra \cite[Problem~4]{CRS18}, which may also be tractable directions to pursue.

\subsection{Commentary and setup}
Interestingly, we avoid analysis of ``Chebyshev's problem'' regarding the greatest prime factor $P^+(f(n))$ of $f(n)$, which is an essential element of the argument in \cite{MR19}. Our approach is to study the product
\[Q(N) = \prod_{n=1}^N |f(n)|.\]
We first analyze the contribution of small primes and linear-sized primes, which we show we can remove and retain a large product. Then we show that each large prime appears in the product a fixed number of times, hence providing a lower bound for the LCM and radical of the LCM. For convenience of our later analysis we write
\[Q(N) = \prod_p p^{\alpha_p(N)}.\]
Note that $\log Q(N) = dN\log N + O(N)$ by Stirling's approximation, if $d$ is the degree of $f$. Finally, let $\rho_f(m)$ denote the number of roots of $f$ modulo $m$.

\subsection*{Remark on notation}
Throughout, we use $g(n)\ll h(n)$ to mean $|g(n)|\le ch(n)$ for some constant $c$, $g(n)\lesssim h(n)$ to mean for every $\epsilon > 0$ we have $|g(n)|\le (1+\epsilon)h(n)$ for sufficiently large $n$, and $g(n)\sim h(n)$ to mean $\lim_{n\to\infty}\frac{g(n)}{h(n)} = 1$. Additionally, throughout, we will fix a single $f\in\mb Z[x]$ that is irreducible over $\mb Q$ and has degree $d\ge 2$. We will often suppress the dependence of constants on $f$. We will also write
\[f(x) = \sum_{i = 0}^d f_ix^i.\]

\subsection{Acknowledgements}
I thank Ze'ev Rudnick, Juanjo Ru\'{e}, and Mark Shusterman for helpful comments, suggestions, and references.

\section{Bounding small primes}\label{sec:bounding-small-primes}
The analysis in this section is very similar to that of \cite[Section~3]{MR19}, except that we do not use the resulting bounds to study the Chebyshev problem. We define
\[Q_S(N) = \prod_{p\le N} p^{\alpha_p(N)},\]
the part of $Q(N) = \prod_{n=1}^N|f(n)|$ containing small prime factors. The main result of this section is the following asymptotic.
\begin{proposition}\label{prop:small-primes}
We have $\log Q_S(N)\sim N\log N$.
\end{proposition}
\begin{remark}
This asymptotic directly implies the earlier stated Equation~\cref{eq:upper-bound}.
\end{remark}
The argument is a simple analysis involving Hensel's Lemma and the Chebotarev density theorem. The Hensel-related work has already been done in \cite{MR19}.
\begin{lemma}[{\cite[Lemma~3.1]{MR19}}]\label{lem:hensel}
Fix $f\in\mb Z[x]$ and assume that it has no rational zeros. Let $\rho_f(m)$ denote the number of roots of $f$ modulo $m$. Then if $p\nmid\on{disc}(f)$ we have
\[\alpha_p(N) = N\frac{\rho_f(p)}{p-1}+O\left(\frac{\log N}{\log p}\right)\]
and if $p\mid\on{disc}(f)$ we have
\[\alpha_p(N)\ll\frac{N}{p},\]
where the implicit constant depends only on $f$.
\end{lemma}
\begin{proof}[Proof of \cref{prop:small-primes}]
We use \cref{lem:hensel}. Noting that the deviation of the finitely many ramified primes from the typical formula is linear-sized, we will be able to ignore them with an error of $O(N)$. We thus have
\begin{align*}
\log Q_S(N) &= \sum_{p\le N}\alpha_p(N)\log p = \sum_{p\le N}N\frac{\log p}{p-1}\rho_f(p) + O\left(\sum_{p\le N}\log N\right) + O(N)\\
&= N\sum_{p\le N}\frac{\log p}{p-1}\rho_f(p) + O(N) = N\log N + O(N),
\end{align*}
using the Chebotarev density theorem alongside the fact that $f$ is irreducible over $\mb Q$ in the last equation (see e.g. \cite[Equation~(4)]{N21}).
\end{proof}

\section{Removing linear-sized primes}\label{sec:removing-linear-sized-primes}
We define
\[Q_{LI}(N) = \prod_{N < p\le DN}p^{\alpha_p(N)},\]
for appropriately chosen constant $D = D_f$. We will end up choosing $D = 1 + d|f_d|$ or so, although any greater constant will also work for the final argument. The result main result of this section is the following.
\begin{proposition}\label{prop:linear-primes}
We have $\log Q_{LI}(N) = O(N)$.
\end{proposition}
In order to prove this, we show that all large primes appear in the product $Q(N)$ a limited number of times.
\begin{lemma}\label{lem:naive-multiplicity}
Let $N$ be sufficiently large depending on $f$, and let $p > N$ be prime. Then
\[\alpha_p(N)\le d^2.\]
\end{lemma}
\begin{proof}
Note that $f\equiv 0\pmod{p}$ has at most $d$ solutions, hence at most $d$ values of $n\in [1,N]$ satisfy $p|f(n)$ since $p > N$. For those values, we see $p^{d+1} > N^{d+1}\ge |f(n)|$ for all $n\in [1,N]$ if $N$ is sufficiently large, and $f$ is irreducible hence has no roots. Thus $p^{d+1}$ does not divide any $f(n)$ when $n\in [1,N]$.

Therefore $\alpha_p(N)$ is the sum of at most $d$ terms coming from the values $f(n)$ that are divisible by $p$. Each term, by the above analysis, has multiplicity at most $d$. This immediately gives the desired bound.
\end{proof}
\begin{proof}[Proof of \cref{prop:linear-primes}]
Using \cref{lem:naive-multiplicity} we find
\[\log Q_{LI}(N)\le d^2\sum_{N < p\le DN}\log p = O(N)\]
by the Prime Number Theorem.
\end{proof}

\section{Multiplicity of large primes}\label{sec:multiplicity-of-large-primes}
Note that \cref{lem:naive-multiplicity} is already enough to recreate \cref{thm:maynard-rudnick}. Indeed, we see that
\[\log\frac{Q(N)}{Q_S(N)} = (d-1)N\log N + O(N)\]
from $Q(N) = dN\log N + O(N)$ and \cref{prop:small-primes}. Furthermore, by definition and by \cref{lem:naive-multiplicity},
\[\frac{Q(N)}{Q_S(N)} = \prod_{p > N}p^{\alpha_p(N)}\le\prod_{p > N, p|Q(N)}p^{d^2}\le\ell_f(N)^{d^2}\le L_f(N)^{d^2}.\]
This immediately gives the desired result (and recreates the constant $\frac{d-1}{d^2}$ appearing in the proof given in \cite{MR19}).

In order to improve this bound, we will provide a more refined analysis of the multiplicity of large primes. More specifically, we will show that we have a multiplicity of $\frac{d(d-1)}{2}$ for primes $p > DN$, with $D$ chosen as in \cref{sec:removing-linear-sized-primes}.
\begin{lemma}\label{lem:refined-multiplicity}
Let $N$ be sufficiently large depending on $f$, and let $p > DN$ be prime, where $D = 1+d|f_d|$. Then
\[\alpha_p(N)\le\frac{d(d-1)}{2}.\]
\end{lemma}
\begin{proof}
Fix prime $p > DN$. As in the proof of \cref{lem:naive-multiplicity}, when $N$ is large enough in terms of $f$, we have that $p^{d+1}$ never divides any $f(n)$ for $n\in [1,N]$. Thus for $1\le i\le d+1$ let $b_i = \#\{n\in [1,N]: p^i|f(n)\}$, where we see $b_{d+1}=0$. Note that
\[\alpha_p(N) = \sum_{i = 1}^d i(b_i-b_{i+1}) = \sum_{i=1}^d b_i.\]
We claim that $b_i\le d-i$ for all $1\le i\le d$, which immediately implies the desired result.

Suppose for the sake of contradiction that $b_i\ge d - i + 1$ for some $1\le i\le d$. Then let $m_1, \ldots, m_{d-i+1}$ be distinct values of $m\in[1,N]$ such that $p^i|f(m)$. Consider the value
\[A = \sum_{j=1}^{d-i+1}\frac{f(m_j)}{\prod_{k\neq j}(m_j-m_k)}.\]
We have from the standard theory of polynomial identities that
\begin{align*}
A &= \sum_{\ell=0}^df_\ell\sum_{j=1}^{d-i+1}\frac{m_j^{\ell}}{\prod_{k\neq j}(m_j-m_k)}\\
&= \sum_{\ell=d-i}^df_\ell\sum_{a_1+\cdots+a_{d-i+1}=\ell-(d-i)}\prod_{j=1}^{d-i+1}m_j^{a_j},
\end{align*}
where the inner sum is over all tuples $(a_1,\ldots,a_{d-i+1})$ of nonnegative integers that sum to $\ell-(d-i)$. Therefore $A\in\mb Z$. Furthermore, since $p^i|f(m_j)$ for all $1\le j\le d-i+1$, we have from the definition of $A$ that
\[p^i|A\prod_{1\le j < k\le d-i+1}(m_j-m_k).\]
Note that each $m_j-m_k$ is nonzero and bounded in magnitude by $N < p$, hence we deduce $p^i|A$.

But from the above formula and the triangle inequality we have
\begin{align*}
|A| &= \left|\sum_{\ell=d-i}^df_\ell\sum_{a_1+\cdots+a_{d-i+1}=\ell-(d-i)}\prod_{j=1}^{d-i+1}m_j^{a_j}\right|\\
&\le\sum_{\ell=d-i}^d|f_\ell|\binom{\ell}{d-i}N^{\ell-(d-i)}\\
&\le (1+|f_d|d^i)N^i
\end{align*}
for sufficiently large $N$ in terms of $f$, using the fact that there are $\binom{\ell}{d-i}$ tuples of nonnegative integers $(a_1,\ldots,a_{d-i+1})$ with sum $\ell-(d-i)$ and that $|m_j|\le N$ for all $1\le j\le d-i+1$.

Thus, as $p > DN\ge (1+|f_d|d)N$, we have
\[|A|\le (1+|f_d|d^i)N^i\le (1+|f_d|d)^iN^i < p^i.\]
Combining this with $p^i|A$, we deduce $A = 0$.

However, we will see that this leads to a contradiction as the ``top-degree'' term of $A$ is too large in magnitude for this to occur. First, we claim that if $1\le i\le d$ and $d-i\le\ell\le d$, then
\begin{equation}\label{eq:am-gm}\tag{4.1}
\frac{\sum_{a_1+\cdots+a_{d-i+1}=\ell-(d-i)}\prod_{j=1}^{d-i+1}m_j^{a_j}}{\sum_{j=1}^{d-i+1}m_j^{\ell-(d-i)}}\in [1,2^d].
\end{equation}
Indeed, note that each $m_j > 0$ and the denominator occurs as a subset of the terms in the numerator, hence the desired fraction is always at least $1$. For an upper bound, simply use the well-known AM-GM inequality. As it turns out, a sharp upper bound for the above is $\frac{1}{d-i+1}\binom{\ell}{d-i}$, which does not exceed $2^d$ for the given range of $i$ and $\ell$.

Next, we see that, using Equation~\cref{eq:am-gm} and the triangle inequality,
\begin{align*}
|A| &= \left|\sum_{\ell=d-i}^df_\ell\sum_{a_1+\cdots+a_{d-i+1}=\ell-(d-i)}\prod_{j=1}^{d-i+1}m_j^{a_j}\right|\\
&\ge|f_d|\sum_{a_1+\cdots+a_{d-i+1}=i}\prod_{j=1}^{d-i+1}m_j^{a_j}-\sum_{\ell=d-i}^{d-1}|f_\ell|\sum_{a_1+\cdots+a_{d-i+1}=\ell-(d-i)}\prod_{j=1}^{d-i+1}m_j^{a_j}\\
&\ge |f_d|\sum_{j=1}^{d-i+1}m_j^i-2^d\sum_{\ell=d-i}^{d-1}|f_\ell|\sum_{j=1}^{d-i+1}m_j^{\ell-(d-i)}\\
&= \sum_{j=1}^{d-i+1}f^\ast(m_j),
\end{align*}
where we define $f^\ast(x) = |f_d|x^i-2^d\sum_{\ell=d-i}^{d-1}|f_\ell|x^{\ell-(d-i)}$. But since $A = 0$ and $f^\ast$ clearly has a global minimum over the positive integers, we immediately deduce that $|m_j|$ for all $1\le j\le d-i+1$ is bounded in terms of some constant depending only on $f$ and $d = \deg f$.

But then, in particular, we also have $|f(m_1)| < C_f$ for some constant $C_f$ depending only on $f$, yet it is divisible by $p > DN$. For $N$ sufficiently large in terms of $f$, this can only happen if $f(m_1) = 0$, but since $f$ is irreducible over $\mb Q$ and $\deg f = d\ge 2$ this is a contradiction! Therefore we conclude that in fact $b_i\le d-i$ for all $1\le i\le d$, which as remarked above finishes the proof.
\end{proof}
We have actually proven something stronger, namely that for this range of $p$ we have at most $d-i$ values $n\in [1,N]$ with $p^i|f(n)$. In particular, this implies that for $p > DN$ we have
\begin{equation}\label{eq:mod-bound}\tag{4.2}
\#\{n\in[1,N]:p|f(n)\}\le d - 1.
\end{equation}

\section{Finishing the argument}\label{sec:finishing-the-argument}
\begin{proof}[Proof of \cref{thm:main}]
The argument is similar to the one at the beginning of \cref{sec:multiplicity-of-large-primes}, but refined. We have
\[\log\frac{Q(N)}{Q_S(N)Q_{LI}(N)} = (d-1)N\log N+O(N)\]
by $Q(N) = dN\log N + O(N)$ and \cref{prop:small-primes,prop:linear-primes}. Furthermore, by definition and by Equation~\cref{eq:mod-bound},
\[\frac{Q(N)}{Q_S(N)Q_{LI}(N)} = \prod_{p>DN}p^{\alpha_p(N)}\le L_f(N)^{d-1}.\]
The inequality comes from the fact that for $p > DN > N$, there are at most $d-1$ values of $n\in [1,N]$ with $p|f(n)$ from Equation~\cref{eq:mod-bound}, and the LCM takes the largest power of $p$ from those involved hence has a power of at least $\frac{\alpha_p(N)}{d-1}$ on $p$. Taking logarithms, we deduce
\[(d-1)\log L_f(N)\ge(d-1)N\log N + O(N),\]
which immediately implies the result since $d\ge 2$.
\end{proof}
\begin{proof}[Proof of \cref{thm:main-rad}]
The argument is essentially identical to the one at the beginning of \cref{sec:multiplicity-of-large-primes}, but with a better multiplicity bound from \ref{lem:refined-multiplicity}. We have
\[\log\frac{Q(N)}{Q_S(N)Q_{LI}(N)} = (d-1)N\log N+O(N)\]
by $Q(N) = dN\log N + O(N)$ and \cref{prop:small-primes,prop:linear-primes}. Furthermore, by definition and by \cref{lem:refined-multiplicity},
\[\frac{Q(N)}{Q_S(N)Q_{LI}(N)} = \prod_{p>DN}p^{\alpha_p(N)}\le\prod_{p>DN,p|Q(N)}p^{\frac{d(d-1)}{2}}\le\ell_f(N)^{\frac{d(d-1)}{2}}.\]
Taking logarithms, we deduce
\[\frac{d(d-1)}{2}\log\ell_f(N)\ge(d-1)N\log N + O(N),\]
which immediately implies the result since $d\ge 2$.
\end{proof}

\section{Discussion}\label{sec:discussion}
We see from our approach that the major obstruction to proving \cref{conj:lcm} is the potential for large prime factors $p > N$ to appear multiple times in the product $Q(N)$. In particular, it is possible to show that \cref{conj:rad} is equivalent to the assertion that
\[\lim_{N\to\infty}\frac{\#\{p\text{ prime}:p^2|Q(N)\}}{\#\{p\text{ prime}:p|Q(N)\}} = 0.\]
Indeed, the bounds we have given are sufficient to show that there are $\Theta(N)$ prime factors of $Q(N)$, of which only $O\left(\frac{N}{\log N}\right)$ are less than $DN$. Therefore the asymptotic size of the LCM is purely controlled by whether multiplicities for large primes in $\left[2,\frac{d(d-1)}{2}\right]$ appear a constant fraction of the time or not (noting that $\log p = \Theta(\log N)$ for these large primes, so that the sizes of their contributions are the same up to constant factors).

Similarly, \cref{conj:lcm} is equivalent to the assertion that
\[\lim_{N\to\infty}\frac{\#\{p\text{ prime}: \exists\, 1\le m<n\le N{:}\,\, p|f(m),p|f(n)\}}{\#\{p\text{ prime}:p|Q(N)\}} = 0.\]

Our bound for \cref{conj:rad} corresponds to using the fact that we can upper bound the multiplicities for all primes $p > DN$ by $\frac{d(d-1)}{2}$. In general, smaller multiplicities other than $1$ could be possible but infrequent, which may be a direction to further approach \cref{conj:lcm} and \cref{conj:rad}.

\bibliographystyle{plain}
\bibliography{main}

\begin{thebibliography}{1}

\bibitem{CRS18}
Pablo Candela, Juanjo Ru\'{e}, and Oriol Serra.
\newblock Memorial to {J}avier {C}illeruelo: a problem list.
\newblock {\em Integers}, 18:Paper No. A28, 9, 2018.

\bibitem{C11}
J.~Cilleruelo.
\newblock The least common multiple of a quadratic sequence.
\newblock {\em Compos. Math.}, 147(4):1129--1150, 2011.

\bibitem{HLQW13}
S.~Hong, Y.~Luo, G.~Qian, and C.~Wang.
\newblock Uniform lower bound for the least common multiple of a polynomial
  sequence.
\newblock {\em C. R. Math. Acad. Sci. Paris}, 351(21-22):781--785, 2013.

\bibitem{HQT12}
S.~Hong, G.~Qian, and Q.~Tan.
\newblock The least common multiple of a sequence of products of linear
  polynomials.
\newblock {\em Acta Math. Hungar.}, 135(1-2):160--167, 2012.

\bibitem{MR19}
J.~Maynard and Z.~Rudnick.
\newblock A lower bound on the least common multiple of polynomial sequences.
\newblock {\em Preprint}, arXiv:1910.13218.

\bibitem{N21}
T.~Nagel.
\newblock Généralisation d'un théorème de {T}chebycheff.
\newblock {\em Journ. de Math.}, 8(4):343--356, 1921.

\bibitem{RZ19}
Z.~Rudnick and S.~Zehavi.
\newblock On {C}illeruelo's conjecture for the least common multiple of
  polynomial sequences.
\newblock {\em Preprint}, arXiv:1902.01102.

\end{thebibliography}

\end{document}